\documentclass[12pt,letterpaper,reqno]{amsart}

\usepackage{comment}
\usepackage{appendix}
\usepackage{graphicx} % Required for inserting images
\usepackage{amsfonts}
\usepackage{amsmath}
\allowdisplaybreaks
\usepackage{amsthm}
\usepackage{xcolor}
\usepackage{mathtools}
\usepackage{fancyhdr,amssymb}
\usepackage{url}
\usepackage{enumerate}

\newcommand\norm[1]{\left\lVert#1\right\rVert}
\newcommand{\Z}{\mathbb{Z}}

\def\R{\mathbb{R}}
\def\<{\langle}
\def\>{\rangle}

\newtheorem{theorem}{Theorem}[section]

\newtheorem{definition}[theorem]{Definition}

\newtheorem{example}[theorem]{Example}
\newtheorem{conjecture}[theorem]{Conjecture}

\begin{document}

\pagestyle{fancy}
\fancyhf{} % Clear all headers and footers
\lhead{} % Left header - empty
\rhead{} % Right header - empty
\cfoot{} % Clear the footer (leave empty if you don't want page numbers)
% Remove header and footer lines
\renewcommand{\headrulewidth}{0pt} % No line in header
\renewcommand{\footrulewidth}{0pt} % No line in footer

\newcommand{\C}{\mathbb{C}}
\newcommand{\supp}{\mathrm{supp}}
\newcommand{\Conj}{\mathrm{Conj}}

%%%%%%%%%%%%%%%%%%%%%%%%%%%%%%%%%%%%%%%%%%%%%%%%%%%%%%%%%%%%%%

\title{Refined Additive Uncertainty Principle}
\author{Ivan Bortnovskyi}\email{ib538@cam.ac.uk}\address{Department of Pure Mathematics and Mathematical Statistics, Univeristy of Cambridge, Cambridge, United Kingdom, CB3 0WA}

\author{June Duvivier}\email{june@duvivier.us}\address{Department of Mathematics, Reed College, Portland, OR 97202}

\author{Alex Iosevich}\email{iosevich@gmail.com}
\address{Department of Mathematics, University of Rochester, Rochester, NY 14627}

\author{Josh Iosevich}\email{joshuaiosevich@gmail.com}\address{Department of Mathematics, Rochester Institute of Technology, Rochester, NY 14623}

\author{Say-Yeon Kwon}\email{sk9017@princeton.edu}\address{Department of Mathematics, Princeton University, Princeton, NJ 08544}

\author{Meiling Laurence}\email{meiling.laurence@yale.edu}\address{Department of Mathematics, Yale University, New Haven, CT 06511}

\author{Michael Lucas}\email{ml2130@cam.ac.uk}\address{Department of Pure Mathematics and Mathematical Statistics, Univeristy of Cambridge, Cambridge, United Kingdom, CB3 0WA}

\author{Tiancheng Pan}\email{tp545@cam.ac.uk}\address{Department of Pure Mathematics and Mathematical Statistics, Univeristy of Cambridge, Cambridge, United Kingdom, CB3 0WA}

\author{Eyvindur Palsson}\email{palsson@vt.edu}\address{Department of Mathematics, Virginia Tech, Blacksburg, VA 24061}

\author{Jennifer Smucker}\email{jennifer21@vt.edu}\address{Department of Mathematics, Virginia Tech, Blacksburg, VA 24061}

\author{Iana Vranesko}\email{yv2@williams.edu}\address{Department of Mathematics, Williams College, Williamstown, MA, 01267}

%\date{Summer 2025}
\maketitle

\begin{abstract}
Signal recovery from incomplete or partial frequency information is a fundamental problem in harmonic analysis and applied mathematics, with wide-ranging applications in communications, imaging, and data science. Historically, the classical uncertainty principles—such as those by Donoho and Stark \cite{DonohoStark}—have provided essential bounds relating the sparsity of a signal and its Fourier transform, ensuring unique recovery under certain support size constraints.

Recent advances have incorporated additive combinatorial notions, notably additive energy, to refine these uncertainty principles and capture deeper structural properties of signal supports. Building upon this line of work, we present a strengthened additive energy uncertainty principle for functions \(f:\mathbb{Z}_N^d\to\mathbb{C}\), introducing explicit correction terms that measure how far the supports are from highly structured extremal sets like subgroup cosets.

We have two main results. Theorem \ref{thm:new_principle} introduces a correction term which strictly improves the additive energy uncertainty principle from \cite{2025additiveenergyuncertaintyprinciple}, provided that the classical uncertainty principle is not satisfied with equality. Theorem \ref{thm:new_recovery} uses the improvement to obtain a better recovery condition. These theorems deliver strictly improved bounds over prior results whenever the product of the support sizes differs from the ambient dimension, offering a more nuanced understanding of the interplay between additive structure and Fourier sparsity. Importantly, we leverage these improvements to establish sharper sufficient conditions for unique and exact recovery of signals from partially observed frequencies, explicitly quantifying the role of additive energy in recoverability.

These results advance the theory of discrete signal recovery by providing stronger, more precise guarantees that bridge harmonic analysis and additive combinatorics, and open new pathways for analyzing sparsity and structure in finite discrete settings.
\end{abstract}

\tableofcontents
\newpage

\section{Introduction}
Let \(f:\mathbb{Z}_N^d\to\mathbb{C}\) be a function on the \(d\)-dimensional module over \(\mathbb{Z}_N:=\mathbb{Z}/N\mathbb{Z}\), where \(N\ge 2\) is an arbitrary positive integer. Our convention is to define the discrete Fourier transform (DFT) of $f$ by \(\hat{f} : \mathbb{Z}_N^d \to \mathbb{C}\) such that for \(m \in \mathbb{Z}_N^d\)
\begin{equation}
    \hat{f}(m) \coloneq N^{-d/2} \sum_{x \in \mathbb{Z}_N^d} f(x) \chi(-m \cdot x),
\end{equation}
where $\displaystyle\chi(t) := e^{\frac{2\pi i t}{N}}$.
We can then recover $f$ using the inverse Fourier transform
\begin{equation}
    f(x) = N^{-d/2} \sum_{m \in \mathbb{Z}_N^d} \hat{f}(m) \chi(m \cdot x).
\end{equation}

The problem of signal recovery is an inverse problem, where we aim to reconstruct the original function $f$ from a given set of observations of its transform $\hat{f}$. This paper aims to improve the conditions that are sufficient for the unique and exact recovery. 
The initial conditions are derived directly from the classical Fourier uncertainty principle. Specifically, Donoho and Stark established the following result, formulated here in arbitrary dimensions.

\begin{theorem}[{\cite{DonohoStark}}]
\label{thm:DS}
Let \( f : \mathbb{Z}_N^d \to \mathbb{C} \) be a finite signal in \(\mathbb{Z}_N^d\) with \( \supp(f) = E \subseteq \mathbb{Z}_N^d \) being the set of non-zero entries. Suppose that the set of unobserved frequencies \(\{f(m)\}_{m \in \mathbb{Z}^d_N}\) is the set \( S \subseteq \supp(\hat{f}) \subseteq \mathbb{Z}_N^d \). Then the signal \( f \) can be recovered uniquely from the observed frequencies if
\begin{equation}
\label{eq:DScondition}
|E| \cdot |S| < \frac{N^d}{2}.
\end{equation}
\end{theorem}

However, Aldahleh-Iosevich-Iosevich-Jaimangal-Mayeli-Pack recently demonstrated that incorporating additive combinatorial structure, specifically through additive energy, yields stronger uncertainty principles. The additive energy of a set $A \subset \mathbb{Z}_N^d$ is defined as:

\begin{definition}[Additive Energy, \cite{2025additiveenergyuncertaintyprinciple}]
Let $A \subset \mathbb{Z}_N^d$. The \emph{additive energy} of $A$, denoted by $\Lambda_2(A)$, is given by
\[
\Lambda_2(A) = \#\{ (x_1, x_2, x_3, x_4) \in A^4 \mid x_1 + x_2 = x_3 + x_4 \}.
\]
\end{definition}

In Theorem 1.8. of \cite{2025additiveenergyuncertaintyprinciple}, the authors showed that additive energy could be incorporated into the uncertainty principle in the following way.

\begin{theorem}[Theorem 1.8, \cite{2025additiveenergyuncertaintyprinciple}] \label{thm:add_uncertainty}
Let $f : \mathbb{Z}_N^d \to \mathbb{C}$ be a nonzero function with support $E = \supp(f)$ and Fourier support $\Sigma = \supp(\hat{f})$. Then
\begin{equation}\label{eqn:add_uncertainty}
    N^d \le |E|\, \Lambda_2(\Sigma)^{1/3},
\end{equation}
and, by symmetry,
\begin{equation*}
    N^d \le |\Sigma|\, \Lambda_2(E)^{1/3}.
\end{equation*}
\end{theorem}

Since the expression is symmetric in $E$ and $\Sigma$, Fourier inversion allows us to interchange the roles of the support and the Fourier support. When $\Lambda_2(\Sigma) < |\Sigma|^3$, Theorem~\ref{thm:add_uncertainty} yields a stronger bound than the standard uncertainty principle. In contrast, if $\Sigma = a + H$ is a coset of a subgroup $H \leq \mathbb{Z}_N^d$, then $\Lambda(\Sigma) = |\Sigma|^3$, and the two principles coincide.

Our main contribution is a strengthened uncertainty principle with explicit correction terms that vanish precisely in the extremal cases. Our first result is the following. 
\begin{equation}\label{eqn:new_uncertainty}
N^d \leq |E| \big(\Lambda_2(\Sigma) - C(E,\Sigma)\big)^{1/3},
\end{equation}
\begin{equation*}
N^d \leq |\Sigma| \big(\Lambda_2(E) - C(\Sigma,E)\big)^{1/3},
\end{equation*}
where $C(E,\Sigma)$ is a non-negative constant depending on $E$ and $\Sigma$, which vanishes exactly when the equality \(N^d = |E|\,|\Sigma|\)
holds. In this case, it follows that both $E$ and $\Sigma$ achieve maximal additive energy, that is,
\[
\Lambda_2(E) = |E|^3 \quad \text{and} \quad \Lambda_2(\Sigma) = |\Sigma|^3.
\]

This maximal additive energy characterizes highly structured sets such as cosets of subgroups in $\mathbb{Z}_N^d$. When $N^d \neq |E|\,|\Sigma|$, the additive energy uncertainty inequalities \eqref{eqn:new_uncertainty} provide a strict improvement over the additive uncertainty principle \eqref{eqn:add_uncertainty} introduced in \cite{2025additiveenergyuncertaintyprinciple}.

Our main result is the following uncertainty principle.

\begin{theorem}[Stronger Additive Uncertainty Principle]\label{thm:new_principle}
Let $f:\Z_N^d\to\C$ be a nonzero function (signal) with support $\supp(f)=E$ and Fourier support $\supp(\hat{f})=\Sigma$.
\begin{equation}
N^d
\leq
|E|
\left(
\Lambda_2(\Sigma)
-|\Sigma|^2
\left(1-\frac{N^d}{|E||\Sigma|}\right)
-|\Sigma|(|\Sigma|-1)
\left(1-\sqrt{\frac{N^d}{|E||\Sigma|}}\sqrt{\frac{\Lambda_2(E)}{|E|^3}}\right)
\right)^{1/3}
\end{equation}
\begin{equation*}
N^d
\leq
|\Sigma|
\left(
\Lambda_2(E)
-|E|^2
\left(1-\frac{N^d}{|E||\Sigma|}\right)
-|E|(|E|-1)
\left(1-\sqrt{\frac{N^d}{|E||\Sigma|}}\sqrt{\frac{\Lambda_2(\Sigma)}{|\Sigma|^3}}\right)
\right)^{1/3}.
\end{equation*}
\end{theorem}

This result fits the general form of the new uncertainty principle stated in \eqref{eqn:new_uncertainty}, with explicit constants given by  
\[
C(E, \Sigma) =  |\Sigma|^2 \left(1 - \frac{N^d}{|E||\Sigma|}\right) 
+|\Sigma| (|\Sigma| - 1) \left(1 - \sqrt{\frac{N^d}{|E||\Sigma|}} \sqrt{\frac{\Lambda_2(E)}{|E|^3}} \right),
\]
and
\[
C(\Sigma, E) =  |E|^2 \left(1 - \frac{N^d}{|E||\Sigma|}\right) 
+ |E| (|E| - 1) \left(1 - \sqrt{\frac{N^d}{|E||\Sigma|}} \sqrt{\frac{\Lambda_2(\Sigma)}{|\Sigma|^3}} \right).
\]
\ \\

This theorem sharpens the additive uncertainty principle by subtracting correction terms $C(E,\Sigma)$ and $C(\Sigma,E)$ that quantify how far the supports $E$ and $\Sigma$ are from this extremal, highly structured case. When the product $|E||\Sigma|$ is strictly greater than $N^d$, the new inequalities provide a stricter bound than the previously known principle in \eqref{eqn:add_uncertainty} from \cite{2025additiveenergyuncertaintyprinciple}.

Let us demonstrate it with the following example:

\begin{example}
    We will work in $\mathbb{Z}_N^2$. Fix an integer $1 \leq m < N$ such that $m\nmid N$ and let the support of our signal be $A$. Define $A$ as follows:\[
    I = \{0, 1,\cdots, m-1\}, \quad A = I\times I \subset \mathbb{Z}_N^2.
    \]
    The additive energy of $A$ is the number of quadruples such that $(x_1,y_1) + (x_2,y_2) = (x_3,y_3) + (x_4,y_4)$, where $(x_i,y_i) \in \mathbb{A}, i \in \{1,2,3,4\}$. This is the same as counting the number of solutions to this system: \[
    x_1+x_2 = x_3+x_4, \quad y_1+y_2 = y_3+y_4, \quad x_i,y_i\in I, i \in \{1,2,3,4\}.
    \]
    Hence, notice that $\Lambda_2(A) = \Lambda_2(I)^2$, where $\Lambda_2(I)$ is the additive energy of $I$.\\
    Let us compute $\Lambda_2(I)$. Let \(r(t) = \#\{(x,y) \in I^2 : x + y = t\}\) for \(t \in \{0, \dots, 2m-2\}\). Then
    \[
    r(t) =
    \begin{cases}
    t + 1, & 0 \le t \le m-1, \\[4pt]
    2m - 1 - t, & m \le t \le 2m - 2.
    \end{cases}
    \]
    Hence
    \[
    \Lambda_2(I) = \sum_{t=0}^{2m-2} r(t)^2
      = 2\sum_{k=1}^{m} k^2 - m^2
      = 2 \cdot \frac{m(m+1)(2m+1)}{6} - m^2
      = \frac{2m^3 + m}{3},
      \]
      \[
      \Lambda_2(A) = \left(\frac{2m^3 + m}{3}\right)^2.
      \]
Let us also notice that $|A| = m^2$.\\
    The additive uncertainty principle introduced in \cite{2025additiveenergyuncertaintyprinciple} yields:\[
    N^2 \leq |\Sigma| (\Lambda_2(A))^{\frac{1}{3}} = |\Sigma|\left(\frac{2m^3 + m}{3}\right)^\frac{2}{3}.
    \]
    Whereas the improved uncertainty principle \ref{thm:new_principle} yields:\begin{align*}
        N^2 \leq |\Sigma| \left(\left(\frac{2m^3 + m}{3}\right)^2 - m^2 \left(1-\frac{N^2}{m^2|\Sigma|}\right) - m^2(m^2-1)\left(1-\sqrt{\frac{N^2\Lambda_2(\Sigma)}{m^2|\Sigma|^4}}\right)\right)^\frac{1}{3}.
    \end{align*}
    Because we supposed $m \nmid N$ and the classical uncertainty principle states $\frac{N^d}{|E||\Sigma|}\leq 1$, then $\frac{N^2}{m^2|\Sigma|} \neq 1$ and $1- \frac{N^2}{m^2|\Sigma|} = \mu > 0$. Therefore, we get the following inequalities: \begin{align*}
        &|\Sigma| \left(\left(\frac{2m^3 + m}{3}\right)^2 - m^2 \left(1-\frac{N^2}{m^2|\Sigma|}\right) - m^2(m^2-1)\left(1-\sqrt{\frac{N^2\Lambda_2(\Sigma)}{m^2|\Sigma|^4}}\right)\right)^\frac{1}{3} \\
        &\leq |\Sigma| \left(\left(\frac{2m^3 + m}{3}\right)^2 - m^2 \mu - m^2(m^2-1)\mu\right)^\frac{1}{3} < |\Sigma|\left(\frac{2m^3 + m}{3}\right)^\frac{2}{3}.
    \end{align*}
\end{example}

Building on the strengthened additive energy uncertainty principle, we derive a sufficient condition for the unique recovery of a signal when certain frequencies are unobserved. This result quantifies how the additive energy of the unobserved frequency set influences recoverability. 

\begin{theorem}[Additive Recovery Condition]\label{thm:new_recovery}
    Suppose frequencies in $S \subseteq \mathbb{Z}_N^d$ are unobserved, and that for all subsets $T \subseteq \mathbb{Z}_N^d$ with $|T| \leq 2|E|$ the additive energy satisfies
    \[
    \Lambda_2(T) \leq K |T|^\alpha,
    \]
    where $K \geq 0$ and $2 \leq \alpha \leq 3$. If
    \begin{align*}
        &|E|^3 \bigg(\Lambda_2(S) - |E|^3 |S| (|S| - 1) \left[ 1 - \sqrt{\frac{K}{(2|E|)^{3-\alpha}}} \sqrt{\frac{N^d}{2|E||S|}} \right] \\
        &\quad - |E|^3 |S|^2 \left( 1 - \frac{N^d}{2|E||S|} \right) \bigg)< \frac{N^{3d}}{8},
    \end{align*}
    then the function $f$ can be uniquely recovered.
\end{theorem}

Therefore, we get better recovery conditions using the stronger additive uncertainty principle (Theorem \ref{thm:new_principle}) than was previously derived from the original additive uncertainty principle established by Aldahleh-Iosevich-Iosevich-Jaimangal-Mayeli-Pack in \cite{2025additiveenergyuncertaintyprinciple}.

We now recall two probabilistic recovery results from Burstein, Iosevich, Mayeli, and Nathan \cite{burstein2025fourierminimizationimputationtime}, which are relevant when the set of missing values is chosen randomly with uniform probability. These results provide theoretically stronger conditions than the one in Theorem \ref{thm:new_recovery} under which a function can be uniquely recovered, either via least squares or using Logan's $L^1$-minimization method (see \cite{Logan1965}, \cite{burstein2025fourierminimizationimputationtime} for details).  

\begin{theorem}[Recovery via least squares, Theorem 1.21 \cite{burstein2025fourierminimizationimputationtime}]\label{thm:probability_recovery_least_squares}
Let $f : \mathbb{Z}_N \to \mathbb{C}$ be supported in $E \subset \mathbb{Z}_N$. Assume that $\{\hat{f}(m)\}_{m \in S}$ are unobserved, where $S$ is a subset of $\mathbb{Z}_N^d$ of size $\lceil N^{2/q} \rceil$ for some $q > 2$, chosen randomly with uniform probability. Then there exists a constant $C(q)$ depending only on $q$ such that with probability $1 - o_N(1)$, if
\begin{equation}\label{eq:thm121_condition}
|E| < \frac{N}{2} \left( \frac{C(q)}{\varepsilon} \right)^{\frac{1}{2 - 2/q}},
\end{equation}
then $f$ can be recovered uniquely, where $\varepsilon = o_N(1)$.
\end{theorem}

\begin{theorem}[Recovery via Logan's method, Theorem 1.23 \cite{burstein2025fourierminimizationimputationtime}]\label{thm:probability_recovery_logan}
Let $S$ be a random subset of $\mathbb{Z}_N^d$ of size $\lceil N^{2d/q} \rceil$ for some $q > 2$. Suppose that $f : \mathbb{Z}_N^d \to \mathbb{C}$ and that the frequencies $\{f_b(m)\}_{m \in S}$ are unobserved. Assume that $f$ is supported in $E \subset \mathbb{Z}_N^d$. If
\begin{equation}\label{eq:thm123_condition}
|E| < N^d \cdot 4 \left( \frac{C(q)}{\varepsilon} \right)^{\frac{2q}{q-2}},
\end{equation}
then $f$ can be recovered using Logan's method with probability at least $1 - \varepsilon$.
\end{theorem}

While these results are theoretically stronger, they also require stronger assumptions on the set of missing frequencies. The refined additive recovery condition in Theorem \ref{thm:new_recovery} applies to non-random missing sets, which is more general than the results obtained in \cite{burstein2025fourierminimizationimputationtime} that require randomness. 

Additionally, even under the assumption that the set of missing frequencies is chosen randomly, Theorems \ref{thm:probability_recovery_least_squares} and \ref{thm:probability_recovery_logan} only hold with high probability, whereas Theorem \ref{thm:new_recovery} provides a deterministic guarantee and incorporates a finer combinatorial structure (additive energy), which can allow recovery in situations where previous sparsity-based bounds fail.

%%%%%%%%%%%%%%%%%%%%%%%%%%%%%%%%%%%%%%%%%%%%%%%%%%%%%%%%%%%%%%%%%%%%%%%%%%%%%%%%%%%%%%%%%%%%%%%%%%%%%%%%%%%%%%%%%%%%%%%%%%%%%%%%%%%%%%

\section{Background on Signal Recovery}
Signal recovery is a fundamental problem in applied mathematics and engineering: given incomplete or partially missing information about a signal, can we reconstruct the original signal exactly? In the discrete Fourier setting, this problem is typically phrased as follows. Let $f: \mathbb{Z}_N^d \to \mathbb{C}$ be a discrete signal, and let $\hat{f}$ denote its discrete Fourier transform. Suppose that the values of $\hat{f}$ are missing on a subset $S \subset \mathbb{Z}_N^d$. The central question is: under what conditions on the signal $f$ and the set of missing frequencies $S$ can the original signal $f$ be recovered exactly from the partial information?

A general theoretical condition that guarantees unique recovery of $f$ was established by Donoho and Stark \cite{DonohoStark}. Suppose we have three functions $f, r, g : \mathbb{Z}_N^d \to \mathbb{C}$ whose Fourier transforms agree on the non-missing frequencies $m \notin S$, and assume that their supports satisfy $|\mathrm{supp}(f)| = |\mathrm{supp}(r)| = |\mathrm{supp}(g)| = |E|$. If the product of the support size of $f$ and the size of the missing frequency set satisfies
\[
|E||S| < \frac{N^d}{2},
\] 
then it follows that $r = g = f$. In other words, under this condition, any two functions that match on the known frequencies and have the same support size must be identical, which establishes a uniqueness guarantee for recovery.

The proof, as given in \cite{DonohoStark}, elegantly leverages the classical uncertainty principle. The core idea is to consider the difference function $h = r - g$ of any two candidate signals $r$ and $g$ that agree with $f$ on the known frequencies. By construction, $\widehat{h}$ is supported on $S$ and $h$ is supported on at most $2|E|$ points. Applying the uncertainty principle to this nonzero function $h$ leads to the inequality $N^d \leq |\text{supp}(h)| , |\text{supp}(\widehat{h})| \leq 2|E||S|$, which contradicts the assumption $2|E||S| < N^d$. This contradiction forces $h \equiv 0$, proving uniqueness.

While Theorem \ref{thm:DS} provides a clean sufficient condition for unique identifiability, it is primarily an existence result. The practical task of actually reconstructing the signal $f$ from its partial Fourier measurements is a central problem in the field of sparse recovery or compressed sensing, which is studied in \cite{candes2005stablesignalrecoveryincomplete,Donoho2006CompressedSensing}. The insight that enabled this field was that the computationally intractable $\ell^0$-minimization (which directly seeks the sparsest solution) can be replaced by its convex relaxation, $\ell^1$-minimization, under certain conditions. The idea is that, given the known frequencies $\hat{f}(m)$ for $m \notin S$, the original signal $f$ can be recovered as the unique minimizer of the $\ell^1$ norm among all functions that agree with $\hat{f}$ on the known frequencies. Formally, the recovery problem is written as
\[
f = \operatorname{argmin}_{g} \|g\|_{L^1(\mathbb{Z}_N^d)} \quad \text{subject to } \hat{g}(m) = \hat{f}(m) \text{ for all } m \notin S.
\]

The guarantee that $\ell^1$ minimization indeed recovers $f$ under the condition $|E||S| < N^d/2$ is a central result also established by Donoho and Stark \cite{DonohoStark}.

As shown in \cite{DonohoStark}, for any function $h$ with $\mathrm{supp}(\widehat{h}) \subseteq S$, the following inequality holds:
\[
\|h\|_{L^1(E)} \le \frac{|E||S|}{N^d} \|h\|_{L^1(\mathbb{Z}_N^d)}.
\]

The recovery proof then proceeds by considering a candidate minimizer $g$ and the difference function $h = g - f$. Since $\widehat{h}$ is supported on $S$, the above inequality applies. Under the assumption $|E||S| < N^d/2$, one can show that the $\ell^1$ norm of $h$ on the complement of the true support $E^c$ must be strictly greater than its norm on $E$. This leads to a contradiction with the assumption that $g$ has a smaller or equal $\ell^1$ norm than $f$, thereby proving that $f$ is the unique minimizer. For the complete and detailed argument, we refer the reader to the original proof in \cite{DonohoStark}.
To illustrate these concepts concretely, consider a simple one-dimensional example:

\begin{example}
Let $f: \Z_4 \to \C$ be defined by its values:
\[
f(0) = 1, \quad f(1) = 0, \quad f(2) = 0, \quad f(3) = 2,
\]
which we denote compactly as $f = (1,0,0,2)$. The support of $f$ is 
\[
\supp(f) = E = \{0,3\} \subset \Z_4.
\]
The discrete Fourier transform (DFT) of $f$ is given by 
\[
\hat{f}(\xi) = \sum_{x=0}^3 f(x) e^{-2\pi i x\xi/4}, \quad \xi \in \Z_4.
\]
Thus, $\hat{f} = (3, 1-2i, -1, 1+2i)$.\\
Suppose during transmission, the Fourier coefficients at frequencies $S = \{1,2\}$ are lost. We attempt to recover $f$ by solving the $\ell^1$ minimization problem:
\[
\min_{g: \Z_4 \to \C} \norm{g}_1 = \min_{g: \Z_4 \to \C} \sum_{x=0}^3 |g(x)|
\]
subject to the constraints $\hat{g}(0) = 3$ and $\hat{g}(3) = 1+2i$.\\
We claim that $f$ is the unique minimizer. To prove this, let $g$ be any function satisfying the constraints. Write $g = f + h$, where $h: \Z_4 \to \C$ is such that $\hat{h}(0) = \hat{h}(3) = 0$. Let $h = (h_0, h_1, h_2, h_3)$ with $h_j = h(j)$. The constraints imply:
\begin{align}
\hat{h}(0) &= h_0 + h_1 + h_2 + h_3 = 0, \label{eq1} \\
\hat{h}(3) &= h_0 - i h_1 - h_2 + i h_3 = 0. \label{eq2}
\end{align}
Subtracting (\ref{eq2}) from (\ref{eq1}) gives:
\[
(h_0 + h_1 + h_2 + h_3) - (h_0 - i h_1 - h_2 + i h_3) = 0 \Rightarrow (1+i)h_1 + (1+i)h_2 = 0.
\]
Since $1+i \neq 0$, we obtain $h_1 + h_2 = 0$, so $h_2 = -h_1$.\\
Substituting $h_2 = -h_1$ into (\ref{eq1}):
\[
h_0 + h_1 - h_1 + h_3 = 0 \Rightarrow h_0 + h_3 = 0 \Rightarrow h_3 = -h_0.
\]
Now substitute $h_2 = -h_1$ and $h_3 = -h_0$ into (\ref{eq2}):
\[
h_0 - i h_1 - (-h_1) + i(-h_0) = h_0 - i h_1 + h_1 - i h_0 = (1-i)h_0 + (1-i)h_1 = 0.
\]
Again, since $1-i \neq 0$, we get $h_0 + h_1 = 0$, so $h_0 = -h_1$.
Thus, the general solution is:
\[
h = (h_0, h_1, h_2, h_3) = (-h_1, h_1, -h_1, h_1) = h_1(-1, 1, -1, 1), \quad h_1 \in \C.
\]
Then,
\[
f + h = (1 - h_1, h_1, -h_1, 2 + h_1).
\]
The $\ell^1$ norm is:
\[
\norm{f+h}_1 = |1 - h_1| + |h_1| + |-h_1| + |2 + h_1| = |1 - h_1| + 2|h_1| + |2 + h_1|.
\]
We now show that for any $h_1 \neq 0$, $\norm{f+h}_1 > \norm{f}_1 = |1| + |0| + |0| + |2| = 3$.\\
Consider the function $\phi: \C \to \R$ defined by:
\[
\phi(h_1) = |1 - h_1| + 2|h_1| + |2 + h_1|.
\]
By the triangle inequality:
\[
\phi(h_1) \geq |(1 - h_1) + (2 + h_1)| + 2|h_1| = |3| + 2|h_1| = 3 + 2|h_1|.
\]
Thus, if $h_1 \neq 0$, then $\phi(h_1) > 3$. Equality in the triangle inequality occurs if and only if all complex numbers point in the same direction and the terms have non-negative real parts when scaled appropriately. More rigorously, equality requires:
\begin{itemize}
\item $1 - h_1 = \lambda (2 + h_1)$ for some $\lambda \geq 0$,
\item $h_1 = \mu (2 + h_1)$ for some $\mu \geq 0$.
\end{itemize}
If $h_1 \neq 0$, these conditions cannot be satisfied simultaneously while maintaining $|1 - h_1| + |2 + h_1| = 3$. Alternatively, one can verify that the subgradient condition for optimality at $h_1 = 0$ is satisfied only when $h_1 = 0$.\\
Therefore, the unique minimizer occurs at $h_1 = 0$, i.e., $g = f$.
\end{example}

These results, combining uniqueness guarantees from uncertainty principles and constructive recovery via $\ell^1$ minimization, form the foundation for much of the modern theory of signal recovery. For a more comprehensive discussion, including extensions to higher dimensions and alternative reconstruction methods, see \cite{IosevichMayeli}.

The incorporation of additive combinatorial structure into this framework represents a significant theoretical advance. Sets with high additive energy correspond to highly structured configurations such as arithmetic progressions or cosets of subgroups, while sets with low additive energy exhibit less additive structure. The key observation is that less structured sets—those with lower additive energy—admit stronger recovery guarantees.

For a set $A \subseteq \mathbb{Z}_N^d$, the normalized additive energy $\Lambda_2(A)/|A|^3$ ranges from $|A|^{-1}$ (for generic sets) to $1$ (for cosets of subgroups). This normalization quantifies how far a set deviates from maximal additive structure, providing a parameter that our strengthened uncertainty principles exploit to improve recovery conditions.

%%%%%%%%%%%%%%%%%%%%%%%%%%%%%%%%%%%%%%%%%%%%%%%%%%%%%%%%%%%%%%%%%%%%%%%%%%%%%%%%%%%%%%%%%%%%%%%%%%%%%%%%%%%%%%%%%%%%%%%%%%%%%%%%

\section{Proofs of Theorem \ref{thm:new_principle}}
\begin{proof}[Proof of Theorem \ref{thm:new_principle}]
Define $1_{x,y,a}=1_E(x)1_E(y)1_E(x+a)1_E(y+a)$. By Fourier inversion, we have
\begin{align*}
\sum_{m\in\Sigma}|\hat{f}(m)|^4
&=
N^{-2d}\sum_{m\in\Sigma}\sum_{x,y,z,w\in E}f(x)\overline{f(y)}f(z)\overline{f(w)}\chi(m\cdot(x-y+z-w))\\
&\leq
N^{-d}
\sum_{\substack{x+z=y+w\\ x,y,z,w\in E}}
f(x)\overline{f(y)}f(z)\overline{f(w)}\\
&\leq
N^{-d}
\sum_{x,y,a\in\Z_N^d}
|f(x)f(y)f(x+a)f(y+a)|1_{x,y,a}
\end{align*}
Therefore,
$$
N^{3d}\sum_{m\in\Sigma}|\hat{f}(m)|^4
\leq
N^{2d}\sum_{x,y,a\in\Z_N^d}|f(x)f(y)f(x+a)f(y+a)|1_{x,y,a}
$$

By Cauchy Schwarz and another application of Fourier inversion, we have
\begin{align*}
    &N^{2d}\sum_{x,y,a\in\Z_N^d}
    |f(x)f(y)f(x+a)f(y+a)|1_{x,y,a}\\
    &\leq
    N^{2d}\sum_{x,y,a\in\Z_N^d}
    |f(x)f(x+a)|^21_{x,y,a}\\
    &\leq
    \sum_{m_1,\dots,m_4}
    |\hat{f}(m_1)\hat{f}(m_2)\hat{f}(m_3)\hat{f}(m_4)|
    \left|
    \sum_{x,y,a\in\Z_N^d}
    \chi(x\cdot (m_1-m_2+m_3-m_4)\chi(a\cdot (m_3-m_4))1_{x,y,a}
    \right|\\
    &=
    \sum_{\substack{m_1,m_2,m_3,m_4\\
    N_1=m_1-m_2\\
    N_2=m_3-m_4}}
    |\hat{f}(m_1)\hat{f}(m_2)\hat{f}(m_3)\hat{f}(m_4)|
    \left|
    \sum_{x,y,a\in\Z_N^d}
    \chi(x\cdot N_1)\chi((x+a)\cdot N_2)
    1_{x,y,a}
    \right|
    \\
    &\leq
    I+II+III,
\end{align*}
where now we split into cases
\begin{itemize}
    \item[I.] $x\neq y$, $a\neq 0$
    \item[II.] $x=y$
    \item[III.] $x\neq y$, $a=0$
\end{itemize}
In what follows, we bound I, II, and III above by terms of the form $(\dots)\sum_{m\in\Sigma}|\hat{f}(m)|^4$ so that we may cancel to recover an improved uncertainty principle.\\
Case I: Notice that the sum
\begin{align*}
&\sum_{\substack{x,y,a\in\mathbb{Z}_N^d x\neq y,a\neq0}}1_E(x)1_E(y)1_E(x+a)1_E(y+a)
\end{align*}
counts the number of triples $(x, y, a)$ with $x \neq y$ and $a \neq 0$ such that all four points $x$, $y$, $x+a$, and $y+a$ lie in $E$. This is equivalent to counting quadruples $(x, y, z, w) \in E^4$ where $z = x + a$ and $w = y + a$ for some nonzero $a$, and $x \neq y$. Since $a = z - x = w - y \neq 0$, the conditions $x \neq y$ and $a \neq 0$ imply that $x \neq z$ and $z \neq w$. Thus, we have the identity:
\begin{align*}
&\{(x,y,a):x,y,x+a,y+a\in E, x\neq y, a\neq0\}\\
&\quad = \{(x,y,z,w)\in E^4 : z = x + a, w = y + a, a \neq 0, x \neq y \}\\
&\quad = \{(x,w,y,z)\in E^4 : x + w = y + z, x \neq z, z \neq w \}.
\end{align*}
This counts all additive quadruples $(x, y, z, w)$ where the pair $(z, y)$ is distinct from both $(x, w)$ and $(w, x)$. The total additive energy $\Lambda_2(E)$ counts all quadruples satisfying $x + w = z + y$. The only quadruples not counted in the above expression are the degenerate cases where $(z, y)$ is identical to $(x, w)$ or $(w, x)$. The number of these trivial quadruples is $|E|^2$ (for $(z,y) = (x,w)$) plus $|E|^2$ (for $(z,y) = (w,x)$), but the $|E|$ quadruples where $x = w$ in both pairs are counted twice. Therefore, the number of non-trivial quadruples is:
\begin{align*}
\Lambda_2(E)-2|E|^2+|E|
\end{align*}
Using the triangle inequality and the above identity, we see that
\begin{align*}
I &= \sum_{m_1,\dots,m_4} |\hat{f}(m_1)\dots\hat{f}(m_4)|
     \left| \sum_{\substack{x,y,a\in\Z_N^d\\x\neq y,a\neq0}}
     \chi(x\cdot N_1)\chi((a+x)\cdot N_2) 1_{x,y,a} \right| \\
&= \sum_{m_1,\dots,m_4} |\hat{f}(m_1)\dots\hat{f}(m_4)|
     \sum_{\substack{x,y,a\in\Z_N^d\\x\neq y, a\neq0}} 1_{x,y,a} \\
&= (\Lambda_2(E)-2|E|^2+|E|)\left(\sum_{m}|\hat{f}(m)|\right)^4 \\
&\leq (\Lambda_2(E)-2|E|^2+|E|)|\Sigma|^3
     \left(\sum_m|\hat{f}(m)|^4\right)
     \quad(\text{H\"older's inequality})
\end{align*}
Case II: 
Observe that 
\begin{align*}
    \sum_{N_1\in\Z_N^d}
    \left|\sum_x\chi(x\cdot N_1)1_E(x)\right|^2
    &=
    \sum_{N_1\in\Z_N^d}
    \sum_{x,\tilde{x}}
    \chi((x-\tilde{x})\cdot N_1)1_E(x)\\
    &=
    |E|N^d
\end{align*}
First, let us establish that for fixed $a,b \in \Z_N^d$ and $m_i \in \Sigma$, we know \[
\left(
    \sum_{\substack{m_1-m_2=a\\m_3-m_4=b}} 
    1
    \right) =  \left( \sum_{m_2 \in \Sigma} \mathbf{1}_{\Sigma}(m_2 + a) \right) \left( \sum_{m_4 \in \Sigma} \mathbf{1}_{\Sigma}(m_4 + b) \right) \leq |\Sigma| \cdot |\Sigma| = |\Sigma|^2
\]
Let $t=x+a$. Then, since $x=y$,
\begin{align*}
    II
    &=
    \sum_{m_1,\dots,m_4}
    |\hat{f}(m_1)\dots\hat{f}(m_4)|
    \left|
    \sum_{x,a\in\Z_N^d}
    \chi(x\cdot N_1)
    \chi((a+x)\cdot N_2)
    1_E(x)1_E(x+a)
    \right|\\
    &=
    \sum_{N_1,N_2\in\Z_N^d}
    \sum_{\substack{m_1-m_2=N_1\\m_3-m_4=N_2}}
    |\hat{f}(m_1)\dots\hat{f}(m_4)|
    \left|
    \sum_x
    \chi(x\cdot N_1)1_E(x)
    \right|
    \left|
    \sum_t
    \chi(t\cdot N_2)1_E(t)
    \right|\\
    &\leq
    \left(\sum_{N_1,N_2\in\Z_N^d}
    \left[
    \sum_{\substack{m_1-m_2=N_1\\m_3-m_4=N_2}}
    |\hat{f}(m_1)\dots\hat{f}(m_4)|
    \right]^2
    \right)^{1/2}\times\\
    &\quad\times
    \left(
    \sum_{N_1,N_2\in\Z_N^d}
    \left|\sum_x\chi(x\cdot N_1)1_E(x)\right|^2
    \left|\sum_t\chi(t\cdot N_2)1_E(t)\right|^2
    \right)^{1/2}\\
    &=
    \left(\sum_{N_1,N_2\in\Z_N^d}
    \left[
    \sum_{\substack{m_1-m_2=N_1\\m_3-m_4=N_2}}
    |\hat{f}(m_1)\dots\hat{f}(m_4)|
    \right]^2
    \right)^{1/2}
    \times
    \left(\sum_{N_1\in\Z_N^d}\left|\sum_x\chi(x\cdot N_1)1_E(x)\right|^2\right)\\
    &=
    |E|N^d\left(\sum_{N_1,N_2\in\Z_N^d}
    \left[
    \sum_{\substack{m_1-m_2=N_1\\m_3-m_4=N_2}}
    |\hat{f}(m_1)\dots\hat{f}(m_4)|
    \right]^2
    \right)^{1/2}\\
    &\leq
    |E|N^d
    \left(
    \sum_{N_1,N_2}\left(\sum_{\substack{m_1-m_2=N_1\\ m_3-m_4=N_2}}
    |\hat{f}(m_1)\dots\hat{f}(m_4)|^2\right)
    \left(
    \sum_{\substack{m_1-m_2=N_1\\m_3-m_4=N_2}}
    1
    \right)
    \right)^{1/2} \\
    &\leq
    |E||\Sigma|N^d
    \left(\sum_{m_1,\dots,m_4}|\hat{f}(m_1)\dots\hat{f}(m_4)|^2\right)^{1/2}\\
    &=
    |E||\Sigma|N^d
    \left(\sum_m|\hat{f}(m)|^2\right)^2\\
    &\leq
    |E||\Sigma|^2N^d\left(\sum_{m}|\hat{f}(m)|^4\right)
    \quad(\text{Holder's inequality})
\end{align*}
Case III: 
We have that

\begin{align*}
    III
    &=
    \sum_{m_1,\dots,m_4}|\hat{f}(m_1)\dots\hat{f}(m_4)|
    \left|
    \sum_{\substack{x,y\in\Z_N^d\\x\neq y}}
    \chi(x\cdot (m_1-m_2+m_3-m_4))1_E(x)1_E(y)
    \right|\\
    &=
    (|E|-1)
    \sum_{m_1,\dots,m_4}
    |\hat{f}(m_1)\dots\hat{f}(m_4)|
    \left|
    \sum_{x\in E}
    \chi(x\cdot (m_1-m_2+m_3-m_4))
    \right|
\end{align*}
Let us look at
\begin{align*}
    &\sum_{m_1,\dots,m_4}|\hat{f}(m_1)\dots\hat{f}(m_4)|
    \left|
    \sum_{x\in E}
    \chi(x\cdot(m_1-m_2+m_3-m_4))
    \right|\\
    &=
    \sum_{M\in\Z_N^d}
    \sum_{m_1-m_2+m_3-m_4=M}
    |\hat{f}(m_1)\dots\hat{f}(m_4)|
    \left|
    \sum_{x\in E}
    \chi(x\cdot M)
    \right|\\
    &\leq
    \left(
    \sum_{M\in\Z_N^d}
    \left(
    \sum_{m_1-m_2+m_3-m_4=M}
    |\hat{f}(m_1)\dots\hat{f}(m_4)|
    \right)^2
    \right)^{1/2}\\
    &\quad\times
    \left(
    \sum_{M\in\Z_N^d}
    \left|
    \sum_{x\in E}\chi(x\cdot M)
    \right|^2
    \right)^{1/2}\\
    &\leq
    \left(
    \sum_{M\in\Z_N^d}
    \left(\sum_{m_1-m_2+m_3-m_4=M}1\right)
    \left(\sum_{m_1-m_2+m_3-m_4=M}|\hat{f}(m_1)\dots\hat{f}(m_4)|^2\right)
    \right)^{1/2}
    \\
    &\quad\times
    |E|^{1/2}N^{d/2}\\
    &\leq
    (\max_M|\{(m_1,m_2,m_3,m_4)\in\Sigma^4\mid
    m_1-m_2+m_3-m_4=M\}|)^{1/2}\\
    &\quad\times
    \left(\sum_{m_1,m_2,m_3,m_4}|\hat{f}(m_1)\dots\hat{f}(m_4)|^2\right)^{1/2}\\
    &\quad\times
    |E|^{1/2}N^{d/2}\\
    &\leq
    (\max_M|\{(m_1,m_2,m_3,m_4)\in\Sigma^4\mid
    m_1-m_2+m_3-m_4=M\}|)^{1/2}\\
    &\quad \times
    |E|^{1/2}|\Sigma|N^{d/2}
    \left(\sum_m|\hat{f}(m)|^4\right)
\end{align*}
Estimating the number of quadruples $$| (m_1, m_2, m_3, m_4) \in \Sigma ^4: m_1 +m_3 = m_2 +m_4 + M|.$$ We can say, that by shifting $m_4$ by M, it is the same as $$ | (m_1, m_2, m_3, m_4) \in \Sigma \times \Sigma \times \Sigma \times (\Sigma +M): m_1 +m_3 = m_2 +m_4 |$$
Now, using bound 5.2 from \cite{2025additiveenergyuncertaintyprinciple}, this quantity is bounded by 
\[
\big(\Lambda_2(\Sigma)^3 \, \Lambda_2(\Sigma + M)\big)^{1/4}
= \Lambda_2(\Sigma)
\tag*{\text{as shifting by $M$ does not change the additive energy of $\Sigma$}}
\]
Hence, combining all three parts, we obtain

\begin{align*}
    N^{3d}\sum_{m\in\Sigma}|\hat{f}(m)|^4
    &\leq I+II+III\\
    &\leq
    \left((\Lambda_2(E)-2|E|^2+|E|)|\Sigma|^3+
    |E||\Sigma|^2N^d+
    |E|^{1/2}(|E|-1)|\Sigma|\Lambda_2(\Sigma)N^{d/2}
    \right)\\
    &\quad\times
    \sum_{m}|\hat{f}(m)|^4.
\end{align*}
By canceling $\sum_{m\in\Sigma}|\hat{f}(m)|^4$ from both sides, rearranging terms, and taking a cube root we recover the improved uncertainty principle
\begin{align*}
    N^d
    \leq
    |\Sigma|
    \left(
    \Lambda_2(E)
    -|E|^2
    \left(1-\frac{N^d}{|E||\Sigma|}\right)
    -|E|(|E|-1)
    \left(1-\sqrt{\frac{N^d}{|E||\Sigma|}}\sqrt{\frac{\Lambda_2(\Sigma)}{|\Sigma|^3}}\right)
    \right)^{1/3}.
\end{align*}
\end{proof}

\section{Proof of Theorem \ref{thm:new_recovery}}
    \begin{proof}[Proof of Theorem \ref{thm:new_recovery}]
    Assume, there exists $g : \mathbb{Z}_N^d \to \mathbb{C}$ such that $g \neq f$ and 
    \[
    \hat{g}(m) = \hat{f}(m) \text{ for } m \notin S \quad \text{and} \quad |\supp(g)| = |\supp(f)| = |E|.
    \]
    Then, let $f = g+ h$, where $h: \mathbb{Z}_N^d \to \mathbb{C}$. Because $h = f-g$, then $|\supp(h)| = |T| = |\supp(f)-\supp(g)| \leq 2|E|$. We also know that because $\hat{g}(m) = \hat{f}(m) \text{ for } m \notin S$, then $\supp(\hat{h}) = Q \subseteq S$. Hence by Theorem \ref{thm:new_principle}, if we exchange $\supp(f)$ for $\supp(\hat{f})$, we get 
    \[
    N^d\leq
|T|
\left(
\Lambda_2(Q)
-|Q|^2
\left(1-\frac{N^d}{|Q||T|}\right)
-|Q|(|Q|-1)
\left(1-\sqrt{\frac{N^d}{|Q||T|}}\sqrt{\frac{\Lambda_2(T)}{|T|^3}}\right)
\right)^{1/3}
    \]
    We want to get a contradiction with a condition of the theorem to prove that our $f$ is unique. 
    Let's take both sides of the equation to the third power: \[
        N^{3d}\leq
|T|^3
\left(
\Lambda_2(Q)
-|Q|^2
\left(1-\frac{N^d}{|Q||T|}\right)
-|Q|(|Q|-1)
\left(1-\sqrt{\frac{N^d}{|Q||T|}}\sqrt{\frac{\Lambda_2(T)}{|T|^3}}\right)
\right)
    \]
    We know that \[
    |T|^3(\Lambda_2(Q) - 2|Q^2| + |Q|) \leq 8|E|^2 (\Lambda_2(S) - 2|S^2| + |S|),
    \]
    because the quantity in the parenthesis on the left represents the number of non-trivial parallelograms in $Q$. Since $Q \subseteq S$, then for $S$ that quantity would be larger. 
    For \begin{align*}
            |T|^3&\left(\frac{|Q|N^d}{|T|} + |Q|^2\sqrt{\frac{N^d}{|Q||T|}}\sqrt{\frac{\Lambda_2(T)}{|T|^3}} + |Q|\left(\sqrt{\frac{N^d}{|Q||T|}}\sqrt{\frac{\Lambda_2(T)}{|T|^3}}\right)\right) \\[10pt]
            &\leq |T|^2|Q|N^d + |T| |Q|^{3/2}\sqrt{N^d\cdot\Lambda_2(T)} + |T||Q|^{1/2} \sqrt{N^d\cdot\Lambda_2(T)}\\[10pt]
            & \leq 4|E|^2|S|N^d + 2|E| |S|^{3/2}\sqrt{N^d\cdot K \cdot (2|E|)^{\alpha}} + 2|E||S|^{1/2} \sqrt{N^d\cdot K \cdot (2|E|)^{\alpha}}. \\
    \end{align*}
    Hence, if we combine the two, we would get \[
    8|E|^3 \left( \Lambda_2(S) - |S|(|S|-1) \left[1-\sqrt{\dfrac{K}{(2|E|)^{(3-\alpha)}}} \sqrt{\dfrac{N^d}{2|E||S|}}\right] - |S|^2\left(1-\frac{N^d}{2|E||S|}\right)\right). 
    \]
    From the statement of the theorem, we know 
    \[
        8|E|^3 \left( \Lambda_2(S) - |S|(|S|-1) \left[1-\sqrt{\dfrac{K}{(2|E|)^{(3-\alpha)}}} \sqrt{\dfrac{N^d}{2|E||S|}}\right] - |S|^2\left(1-\frac{N^d}{2|E||S|}\right)\right) < N^{3d}.
    \]
    Hence, combining all the inequalities we get \[
    N^{3d} < N^{3d},
    \]
    which is a contradiction, which means $f$ is unique.
    \end{proof}

\section{Future work}
\noindent Several directions for further research arise naturally from our results.
\medskip \\
\noindent\textit{Extension to Gowers $U^k$ Norms.}
A natural extension of the additive energy uncertainty principle would be to use Gowers $U^k$ norms.
\begin{definition}[Gowers $U^k$ norm]\label{def:Gowers norm}
    Let $k\geq 2$ and $f:\Z_N^d\to\C$. Then, $\|f\|_{U^k}$ is the unique positive real $2^k$-th root of 
\[
\|f\|_{U^k}^{2^k}
=
\frac{1}{N^{d(k+1)}}
\sum_{x\in\Z_N^d}
\
\sum_{\substack{h_1,\dots,h_k\in\Z_N^d\\h=(h_1,\dots,h_k)}}
\
\prod_{w\in\{0,1\}^k}
\Conj^{|w|}f(x+w\cdot h)
\]
where $\Conj^\ell(z)$ denotes $\ell$-fold conjugation.
\end{definition}
It is not immediately obvious that the right hand side is a positive real number, without which this definition is not valid. This can be proven using the inductive definition of the $U^k$ norm and can be found in \cite[Chapter 11]{Tao-Vu-2006}. 
Additive energy can be related to the $U^2$ norm via the formula $\Lambda(S)=N^{3d}\|1_S\|_{U^2}^4$. Rewriting the additive energy uncertainty principle using this formula, we obtain
\[
1\leq |E|\cdot\|1_\Sigma\|_{U^2}^{4/3}.
\]
This suggests the following formula.
\begin{conjecture} 
Let $f:\Z_N^d\to\C$ be a signal with support $E$ and Fourier support $\Sigma$. Then,
\begin{equation}
    1\leq|E|\cdot\|1_\Sigma\|_{U^k}^{2^k/(k+1)}.
\end{equation}
\end{conjecture}
Let $1_{y+H}$ be the indicator function of a coset $y+H$ of a subgroup $H\leq \Z_N^d$. Then, $\supp(1_{y+H})=y+H$ and $\supp(\hat{1}_{y+H})=H^\perp=\{m\in\Z_N^d\mid m\cdot x=0\ \forall x\in H\}$. Then, using the definition of the Gowers $U^k$ norm, it is straightfowards to check $\|1_{y+H}\|_{U^k}^{2^k/(k+1)}=\frac{|H|}{N^d}$ and $\|1_{H^\perp}\|_{U^k}^{2^k/(k+1)}=\frac{|H^\perp|}{|N^d|}$. Hence, 
\[
|\supp(1_{y+H})|\cdot\|1_{H^\perp}\|_{U^k}^{2^k/(k+1)}=|H|\cdot\frac{|H^\perp|}{N^d}=1
\]
and, similarly, $|\supp(\hat{1}_{y+H})|\cdot\|1_{y+H}\|_{U^k}^{2^k/(k+1)}=|H^\perp|\frac{|H|}{N^d}=1$. Therefore, the conjecture holds with equality when $f(x)=1_{y+H}(x)$ or $\hat{f}(x)=1_{y+H}(x)$.
\medskip \\
\noindent\textit{Higher-Order Additive Energies.}
Another approach is inspired by the identity $\Lambda_2(E) = |E|^2$. It is natural to investigate uncertainty principles involving higher-order quantities such as 
\[
    \Lambda_{k+1}(E) - \Lambda_k(E)(\dots),
\]
which enumerate non-degenerate $(k+1)$-dimensional parallelepipeds. Beyond the classical additive energy, one may consider alternative frameworks such as the number of solutions to 
\[
    a+b+c+d = e+f+g+h,
\]
which arise from the Fourier expansion of 
\[
    \sum_{x \in \mathbb{Z}_N^d} |\hat{f}(x)|^8.
\]
These higher-order energies could lead to new structural parameters for uncertainty principles.\medskip \\
\noindent\textit{Algorithmic Extensions.}
On the algorithmic side, while $\ell_1$- and $\ell_2$-minimization methods are standard in compressed sensing, it is natural to ask whether analogous recovery guarantees can be obtained via minimization in the $U_k$-norm. Developing such algorithms would directly leverage higher-order additive structure.
\medskip \\
\noindent\textit{Broader Extensions.}
More broadly, extensions to other finite abelian groups and eventually to continuous domains remain open problems. The adaptation to noisy measurements, rather than complete erasures, is also essential for applications in signal processing.
\medskip \\
In summary, the incorporation of higher-order additive energies, alternative combinatorial invariants, and $U_k$-based algorithms may lead to substantial refinements of uncertainty principles and recovery guarantees.

\section{Acknowledgements}
This paper was written as part of the SMALL REU Program 2025. We appreciate the support of Williams College, Yale University and the Finnerty Fund, as well as funding from NSF grant DMS-2241623. The first listed author was supported by The Winston Churchill Foundation of the United States. The seventh and eighth listed authors were supported by Dr. Herchel Smith Fellowship Fund. A.I. was supported in part by the National Science Foundation under grants no. HDR TRIPODS-1934962, NSF DMS 2154232, and NSF DMS 2506858 during work on this paper.

\bibliographystyle{alpha}
\bibliography{biblio}
\ \\
\end{document}